\documentclass[twoside]{article}
\usepackage[english]{babel}
\usepackage{cite}
\usepackage{setspace}
\usepackage{amssymb,amsmath,amsthm}
\usepackage{longtable}
\usepackage{nicefrac}
\usepackage[justification=justified,singlelinecheck=false]{caption}
\usepackage{color}
\usepackage[section]{placeins}

\usepackage[margin=1in]{geometry}
\newtheorem{defn}{Definition}[section]
\newtheorem{thm}{Theorem}[section]

\newtheorem{lem}{Lemma}[section]
\newtheorem{remark}{Remark}[section]
\newtheorem{corollary}{Corollary}[section]
\begin{document}
\doublespace
%\internallinenumbers
%\linenumbers*
\begin{center}
{\Large \textbf{Generalized Moufang loops with $\alpha$-elasticity property}}\\[5mm]

{ A. O. Abdulkareem$^a$, J. O. Adeniran$^b$, A. A. A. Agboola$^c$ and G. A. Adebayo$^d$}

$ \ ^{a,b,c}$ Department of Mathematics, Federal University of Agriculture Abeokuta, Nigeria.\\[0pt]

$\ ^{a}$ Department of Mathematical Sciences, University of Africa Toru Orua, Bayelsa, Nigeria.\\

$ \ ^{d}$ Department of Physics, Federal University of Agriculture Abeokuta, Nigeria.\\[0pt]

$  ^{a}$afeezokareem@gmail.com \ $^{b}$ekenedilichineke@yahoo.com,\\
$^c$aaaola2003@yahoo.com \ $^d$adebayo@physics.unaab.edu.ng\\
\end{center}

\begin{abstract} 
 \noindent In this study, the problem of existence of universally flexible loop that is not middle Bol is partially solved. $\alpha$-elasticity property is introduced for generalized Moufang loops. Necessary and sufficient conditions for $\alpha$-elastic generalized Moufang loops to be universal are given. Using these conditions, some properties of generalized Moufang loops are obtained. Condition under which the generalized Moufang loop is an abelian group is stated.
\end{abstract}

\thispagestyle{empty}
\section{Introduction}
A Moufang loop is a loop satisfying the Moufang identities $(xy\cdot z)y = x(y\cdot zy)$. Other identities defining Moufang loops are $yz\cdot xy = y(zx\cdot y)$ and $(yz\cdot y)x = y(z\cdot yx)$. They are closely related to the class of Bol loops \cite{Rob1}. Bol loops are characterized by two identities namely; the left Bol identity $y(z.yx)=(y.zy)x$ and the right Bol identity $(xy.z)y=x(yz.y)$. When the two identities hold in a loop, the loop is said to be Moufang. 
%After Robinson, many authors from different parts of the world have studied several properties of Bol loops. Burn \cite{BURN1} proved that, for any prime p, a Bol loop of order $2p$ or of order $p^{2}$ is necessarily a group, and showed that there exist exactly six non-associative Bol loops of order $8$.
%In \cite{BURN2}, it had been proved that for any odd prime $p$, there are exactly two nonassociative, non-Moufang, (right) Bol loops of order $4p$. In \cite{BURN3} Burn gave a unique construction of a Bol loop of order $2p^{2}$, and proved that a Moufang loop with some of the properties of this Bol loop is necessarily a group. %Several other constructions and studies on Bol loops can be found in \cite{CHEIN1,CHEIN2,CHEIN3,KP,LOG,KPV,NAGY1,NAGY2,NAGY3,FKP,SHARMA1}. 

%The notion of half Bol loops \cite{SHASAB1,SHASAB2}, a generalization of Bol loops, was introduced and defined by $(y^{\alpha}\cdot zy)x=y^{\alpha}(z\cdot yx)$. It is easily seen that the generalization comes from the left Bol identity. Around the same period Ajmal (see \cite{ADESOLA}) introduced and studied 
A loop is a generalized Bol loop if it satisfies $(xy\cdot z)y^{\alpha}=x(yz\cdot y^{\alpha})$ and a generalized Moufang loop, $Q$, if $(xy\cdot z)y^{\alpha}=x(y\cdot zy^{\alpha})$ holds, where $\alpha:Q\longrightarrow Q$ is a self map. So, generalized Moufang loop is a generalized Bol loop \cite{AJI} that is elastic, $x(yx)=(xy)x$. 
%Some of the results obtained so far in the theory of generalized Bol loops can be found in \cite{ADESOLA,ADEAKI,AJI,ALA}. 
The question ``Is there a finite universally elastic loop that is not middle Bol" posed by Kinyon in LOOPS'03, Prague 2003 is partially answered in affirmative and our study suggest search for such loops among (generalized) Moufang loops whose exponent is not $2$.

%The question now is what happens to the generalized identity resulting from non proper right Bol loops, $(xy\cdot z)y=x(y\cdot zy)$ \cite{NAGY1} (that is right Bol loops for which the elastic identity $x(yx)=(xy)x$ holds)? It is this question that motivated us to introduce the elastic identity in generalized context. The non proper right Bol identity is one of the defining identity of Moufang loop and thus, we refer to the generalized loops resulting from this identity as generalized Moufang loops. 

The universality of the elasticity property of loops was studied by Syrbu in \cite{PS1} and middle Bol loops emerge has a consequence of loops with universal elasticity property. Although, it has been reported that middle Bol loops were introduced by Belousov in 1967 and were later studied by Gwaramija in 1971\cite{PS2}.

The current study introduces $\alpha$-elasticity property for generalized Bol loops. Some authors prefer flexibility to elasticity but we shall drop flexibility for elasticity throughout this paper. Conditions of Necessity and sufficiency under which $\alpha$-elasticity property of generalized Moufang loops is universal are established. Using the universal conditions, and in some cases, with the newly introduced right and left $\alpha$-alternative property for generalized Moufang loops, some properties of generalized Moufang loops are studied. 

The paper is organized as follows. Section \ref{SEC:2} gives basic concepts and definitions of different terminologies used throughout the paper. Section \ref{SEC:3}, \ref{SEC:4} and \ref{SEC:5} contain the main results of this paper. It is established that nuclei of generalized Bol loops with universal $\alpha$-elasticity do not coincide though it is an inverse property loop. The right and middle nucleus of right inverse property loops coincide but this is not true for right inverse property generalized Moufang loops with universal law of $\alpha$-elasticity. However, left and middle nucleus do coincide in right inverse generalized Moufang loops with universal law of $\alpha$-elasticity.
\section{Basic concepts and Definitions}\label{SEC:2}
\begin{defn}
A groupoid $(L, \cdot)$ consists of a set $L$ together with a binary operation $\cdot$ on $L$.
\end{defn} 
\begin{defn}
For $x \in L$, define the left, respectively right,
translation by $x$ by $L_{x}(y) = x\cdot y$, respectively $R_{x}(y) = y\cdot x$, for all $y\in L$.
\end{defn} 
\begin{defn}
A quasigroup is a groupoid $(L, \cdot)$ with a binary operation $\cdot$ such that for each $a, b\in L$ the equations $a\cdot x = b$ and $y\cdot a = b$ have unique solutions $x, y\in L$.
\end{defn} 
\begin{defn}
A loop is a quasigroup with identity element. A left loop in which all right translations are bijective is also called a loop. 
\end{defn}
For basic facts on loops, we refer the reader to \cite{B, P}. 
\begin{defn}
A loop satisfying the right Bol identity $$((xy)z)y = x((yz)y)$$
%$$(x(yx))z = x(y(xz))$$
or equivalently
$$R_{y}R_{z}R_{y}=R_{(yz)y}$$
%$$L(x(yx))=L(x)L(y)L(x)$$
for all $x, y, z \in L$ is called a right Bol loop. A loop satisfying the mirror identity $(x(yx))z = x(y(xz))$ for all $x, y, z\in L$ is called a left Bol loop, and a loop which is both left and right Bol is a Moufang loop.
\end{defn} 
\begin{defn}\label{Def. GBL}
A loop is a generalized Moufang loop if it satisfies the relation $$(xy\cdot z)y^{\alpha}=x(y\cdot zy^{\alpha})$$
where $y^{\alpha}$ is the image of $y$ under some mapping, $\alpha$, of the loop onto itself.
\end{defn}
Observe that the generalized Moufang identity is coming from the right Bol identity for which the elastic property holds. 
%The dual of the relation in Definition \ref{Def. GBL} is given by $$(y^{\alpha}z\cdot y)x=y^{\alpha}(z\cdot yx)$$ and it is called half-Bol loops.
\begin{defn}
A loop $(G,\cdot)$ is called a left inverse property loop if it satisfies the left inverse property (LIP) given by: $$x^{\lambda}(xy)=y.$$
\end{defn} 
\begin{defn}
A loop $(G,\cdot)$ is called a right inverse property loop if it satisfies the right inverse property (RIP) given by: $$(yx)x^{\rho}=y.$$
\end{defn} 
\begin{defn}
A loop is called an IP loop if it is both LIP-loop and RIP-loop.
\end{defn}
\begin{defn}
Let $(G,\cdot)$ be a loop, the left nucleus, $N_\lambda$, the middle nucleus, $N_\mu$  and the right nucleus, $N_\rho$ are defined as follows:
\begin{align*}
N_\lambda &=\{x\in G|x\cdot yz=xy\cdot z \ \forall \ y,z\in G\} \\
N_\mu &=\{y\in G|x\cdot yz=xy\cdot z \ \forall \ x,z\in G\} \\
N_\rho &=\{z\in G|x\cdot yz=xy\cdot z \ \forall \ x,y\in G\}
\end{align*}
\end{defn}
\noindent The following are few results on the nucleus of generalized Bol loops.
\begin{thm}[\cite{ALA}]\label{GBLN}
If $(G,\cdot)$ is a generalized Bol loop, then $y^{\alpha}\in N_{\rho}$ if and only if $y\in N_{\mu}$. 
\end{thm}

\begin{thm}[\cite{ALA}]
If $(G,\cdot)$ is a generalized Bol loop with $f\in G$, and let $u\circ v=uR_{(g)^{-1}}\cdot vL_{(f)^{-1}}$ for all $u,v\in (G,\circ)$. If $f^{\alpha}\in N_{\rho}$ of $(G,\cdot)$, then $(G,\cdot)$ and $(G,\circ)$ are isomorphic.
\end{thm}
\begin{thm}[\cite{ALA}]
If $(G,\cdot)$ is a generalized Bol loop with $f\in G$, and let $u\circ v=uR(f)\cdot vL(f)^{-1}$ for all $u,v\in (G,\circ)$. If $f\in N_{\mu}$ of $(G,\cdot)$, then $(G,\cdot)$ and $(G,\circ)$ are isomorphic.
\end{thm}

\subsection{Universality Of Loops}
An identity is said to be universal for a loop if it holds in the loop and each of its principal isotope.
Consider $(G,\cdot )$ and $(H,\circ )$ being two distinct
loops. Let $\kappa, \beta$ and $\gamma$ be three bijective mappings, that map $G$ onto $H$. The triple $\theta =(\kappa, \beta, \gamma)$ is called an isotopism of $(G,\cdot )$ onto $(H,\circ )$ if and only if
$$x\kappa\circ y\beta=(x\cdot y)\gamma~\forall~x,y\in G.$$
So, $(H,\circ )$ is called a loop isotope of $(G,\cdot )$.\\
Similarly, the triple
$$\theta^{-1}=(\kappa, \beta, \gamma)^{-1}=(\kappa^{-1},\beta^{-1},\gamma^{-1})$$
is an isotopism from $(H,\circ )$ onto $(G,\cdot )$ so that $(G,\cdot )$ is also called a loop isotope of $(H,\circ)$. Hence, both are said to be isotopic to each other. If one of two isotopic groupoids is a quasigroup, then both are quasigroups, but the same statement is not true if two quasigroups are isotopic and one is a loop. This fact makes it possible and
reasonable to study and consider quasigroups as isotopes of groups.

An isotopism $(\kappa, \beta, \gamma)$ of $(G,\cdot)$ onto $(H,\circ)$ is a principal isotopism if and only if $G=H$ and $\gamma$ is an identity map. Hence,
\begin{defn}
Let $\kappa$ and $\beta$ be permutation of $G$ and let $i$ denote the identity map on $G$. Then $(\kappa, \beta, i)$ is a principal isotopism of a groupoid $(G,\cdot)$ onto a groupoid $(G,\circ)$ means that $(\kappa, \beta, i)$ is an isotopism of $(G,\cdot)$ onto $(G,\circ)$.
\end{defn}
A groupoid $(H,\circ)$ is a principle isotope of a groupoid $(G,\cdot)$ provided $G=H$ and there exist permutations $\kappa$ and $\beta$ of $G$ so that $x\kappa\circ y\beta=x\cdot y$ for all $x,y\in G$.

Now let $(G,\cdot)$ be any quasigroup and let $f$ and $g$ be any element in $G$ (not necessarily distinct). Since $(G,\cdot)$ is a quasigroup, the map $L_{f}$ and $R_{g}$ are permutation of the set $G$. Now define $$x\circ y=xR^{-1}_{g}\cdot yL^{-1}_{f}$$ for all $x,y\in G$. Then $(G,\circ)$ is a principal isotope of $(G,\cdot)$. Clearly $(R_{g},L_{f},i)$ is a principal isotopy of $(G,\cdot)$ onto $(G,\circ)$ and so $(G,\circ)$ is a quasigroup. Since $f\cdot g$ is the two sided identity element of $(G,\circ)$, $(G,\circ)$ is a loop.

The importance of principal isotopy lies in the fact that up to isomorphism the principal isotopy of a groupoid $(G,\cdot)$ account for all the isotopes of $(G,\cdot)$. This explains why principal isotope is employed in the proof of universality of a property in quasigroups and loops theory.
\section{The universality of $\alpha$-elasticity property}\label{SEC:3}
From the generalized Moufang identity $(xy\cdot z)y^{\alpha}=x(y\cdot zy^{\alpha})$, on setting $x=1$, we have $$(y\cdot z)\cdot y^{\alpha}=y\cdot (z\cdot y^{\alpha}).$$ This new identity is called $\alpha$-elasticity property of generalized Moufang loop.\\
One of the important concepts in the theory of quasigroups and loops is the universality of a property. This is exemplified in the theory of algebraic nets wherein all loops which coordinate the same net are isotopic between themselves and the identities which follow from the closure conditions of this net are universal for each of these loops \cite{PS1}. Bol and Moufang identities are examples of identities that are universal for loops \cite{P, PS1}. In the following theorem, we shall use the concept of principal isotopy discussed in the previous section to prove that $\alpha$-elasticity is universal for generalized Moufang loops.
\begin{thm}
Let $G$ be a generalized Moufang loop $Q(\cdot,/,\backslash)$ with $\alpha$-elasticity identity, $(y\cdot z)\cdot y^{\alpha}=y\cdot (z\cdot y^{\alpha})$. Then, the identity is universal for the generalized Moufang loop if and only if the following conditions are satisfied
\begin{equation}\label{UNI1}
(yz/x)(b\backslash y^{\alpha}x)= y(b\backslash\left[(bz/x)(b\backslash y^{\alpha}x)\right])
\end{equation}
\begin{equation}\label{UNI2}
(\left[(by/x)(b\backslash zx)\right]/x)y^{\alpha}= (by/x)(b\backslash zy^{\alpha})
\end{equation}
%
%\begin{align*}
%(yz/x)(b\backslash y^{\alpha}x)&= y(b\backslash\left[(bz/x)(b\backslash y^{\alpha}x)\right]) &\\
%(\left[(by/x)(b\backslash zx)\right]/x)y^{\alpha}&= (by/x)(b\backslash zy^{\alpha})
%\end{align*} 
\end{thm}
\begin{proof}
Let $Q(\cdot,/,\backslash)$ be a generalized Moufang loop with universal $\alpha$-elasticity. Let $Q(\circ)$ denote a principal isotope of $Q(\cdot)$, that is, $x\circ y=R^{-1}_{a}x\cdot L^{-1}_{b}y,$ where $(R^{-1}_{a}, L^{-1}_{b}, I)$ is the principal isotopy, $R^{-1}_{a}$ and $L^{-1}_{b}$ are right and left division by the element $a$, and $b$, respectively.
Thus, $$R^{-1}_{a}(R^{-1}_{a}y\cdot L^{-1}_{b}z)\cdot L^{-1}_{b}y^{\alpha}=R^{-1}_{a}y\cdot L^{-1}_{b}(R^{-1}_{a}z\cdot L^{-1}_{b}y^{\alpha}).$$
On replacing $a$ by $x$, we obtain
$$((y/x\cdot b\backslash z)/x)\cdot (b\backslash y^{\alpha})= y/x\cdot (b\backslash(z/x\cdot b\backslash y^{\alpha}))$$
$$(y/x\cdot b\backslash z)(b\backslash y^{\alpha})= y/x\cdot (b\backslash(z/x\cdot b\backslash y^{\alpha}))x$$
$$(y/x\cdot b\backslash z)(b\backslash y^{\alpha})x= y\cdot (b\backslash(z/x\cdot b\backslash y^{\alpha}))x$$
\begin{equation*}
(yz/x)(b\backslash y^{\alpha}x)= y(b\backslash\left[(bz/x)(b\backslash y^{\alpha}x)\right]),
\end{equation*}
which gives the equation (\ref{UNI1}).\\
and 
$$((y/x\cdot b\backslash z)/x)\cdot (b\backslash y^{\alpha})= y/x\cdot (b\backslash(z/x\cdot b\backslash y^{\alpha}))$$
$$(((by/x)\cdot b\backslash z)/x)y^{\alpha}= (by/x)(z/x\cdot b\backslash y^{\alpha})$$
\begin{equation*}
(\left[(by/x)(b\backslash zx)\right]/x)y^{\alpha}= (by/x)(b\backslash zy^{\alpha}),
\end{equation*}
which gives the equation (\ref{UNI2}).\\
Therefore, the notion of $\alpha$-elasticity is universal for generalized Moufang loops if and only if the identities (\ref{UNI1}) and (\ref{UNI2}) hold in the primitive generalized Moufang loop $Q(\cdot,/,\backslash)$
\end{proof}
\begin{thm}\label{thm3.2}
If $Q(\cdot)$ is a generalized Moufang loop with universal $\alpha$-elasticity then
\begin{enumerate}
\item[i.] $y^{\alpha}\in N_{\rho}$ if and only if $y\in N_{\lambda}$
\item[ii.] if $y$ is a generalized Moufang element of the loop $Q(\cdot)$, then $y\in N_{\mu}$ if and only if $y^{\alpha}\in N_{\rho}$. 
\end{enumerate}
\end{thm}
\begin{proof}
i. From (\ref{UNI2}), let $b=e$ where $e$ is the unit of $Q(\cdot)$, $y$ and $z$ be replaced by $by$ and $zx$ respectively. Then $$(y/x)\cdot zy^{\alpha}=(\left[(y/x)\cdot zx\right]/x)y^{\alpha}$$
If $y^{\alpha}\in N_{\rho}$ in the above equation, then
$$(y/x)\cdot z=\left[(y/x)\cdot zx\right]/x$$ or $$yz\cdot x=y\cdot zx$$
which implies that $y\in N_{\lambda}$.\\
Conversely, taking $x=e$ in (\ref{UNI1}), we have 
$$(yz)\cdot (b\backslash y^{\alpha})= y(b\backslash\left[bz\cdot (b\backslash y^{\alpha})\right]).$$ If $y\in N_{\lambda}$, then 
$$z\cdot (b\backslash y^{\alpha})= b\backslash\left[bz\cdot (b\backslash y^{\alpha})\right]$$ Or 
$$b\cdot zy^{\alpha}=bz\cdot y^{\alpha}$$
which implies that $y^{\alpha}\in N_{\rho}$.\\

ii. An element $y$ of the loop $Q(\cdot)$ is called a generalized Moufang element if the equality $(xy\cdot z)y^{\alpha}=x(y\cdot zy^{\alpha})$ holds $\forall \; x,z\in Q$.\\
Suppose $y$ is a generalized Bol element of the loop $Q(\cdot)$ and $y\in N_{\mu}$, middle nucleus of $Q(\cdot)$,
$$N_{\mu}=\left\{y\in Q: xy\cdot z=x\cdot yz \; \forall x,z\in Q\right\}.$$
Then, $$(xy\cdot z)y^{\alpha}=(x\cdot yz)y^{\alpha}=x(y\cdot zy^{\alpha})=x(yz\cdot y^{\alpha}).$$
So, $$(x\cdot yz)y^{\alpha}=x(yz\cdot y^{\alpha}).$$
On replacing $yz$ by $z$, we obtain $$xz\cdot y^{\alpha}=x\cdot zy^{\alpha},$$
which implies that $y^{\alpha}\in N_{\rho}$.
\end{proof}
\begin{thm}\label{thm3.3}
Let $Q(\cdot)$ be a generalized Moufang loop with universal $\alpha$-elasticity such that $y^{\alpha}b\cdot b^{\alpha}=(yz)^{\alpha}\cdot z$ and $\alpha$ is a homomorphic self map. Then, $Q(\cdot)$ is a generalized Moufang loop if and only if it a middle generalized Bol loop.
\end{thm}
\begin{proof}
Let $z$ be replaced by $y\backslash z$ in (\ref{UNI1}).\\
$$y((y\backslash z)/x)(b\backslash y^{\alpha}x)=y(b\backslash\left[b((y\backslash z)/x)(b\backslash y^{\alpha}x)\right])$$
$$\left[b(y\backslash z)/x\right](b\backslash y^{\alpha}x)=b(y\backslash\left[(z/x)(b\backslash y^{\alpha}x)\right]).$$
Now replace $z$ by $zx$ above
$$\left[b(y\backslash zx)/x\right](b\backslash y^{\alpha}x)=b(y\backslash\left[z(b\backslash y^{\alpha}x)\right]),$$
or if $y$ is replaced by $zx$,
\begin{equation}\label{MGBL}
(b/x)(b\backslash (zx)^{\alpha}x)=b((zx)\backslash \left[z(b\backslash (zx)^{\alpha}x)\right])
\end{equation}
\begin{equation}
(b/x)(b\backslash z^{\alpha}x^{\alpha}\cdot x)=b(zx\backslash \left[z(b\backslash z^{\alpha}x^{\alpha}\cdot x)\right])
\end{equation}
Now let $z^{\alpha}x^{\alpha}\cdot x=z^{\alpha}b\cdot b^{\alpha}$ $\forall \; z,b,x\in Q$,
\begin{align*}
%z^{\alpha}b^{\alpha}\cdot b&=z^{\alpha}x^{\alpha}\cdot x&\\
z^{\alpha}b\cdot b^{\alpha}&=z^{\alpha}x^{\alpha}\cdot x&\\
z^{\alpha}b&=(z^{\alpha}x^{\alpha}\cdot x)/ b^{\alpha} &\\
b&=(z^{\alpha}x^{\alpha}\cdot x)/(z^{\alpha}\backslash b^{\alpha})
\end{align*}
put the last equality in (\ref{MGBL}), we have
$$(b/x)(z^{\alpha}\backslash b^{\alpha})=b(z^{\alpha}x\backslash b^{\alpha}).$$
The converse of the theorem follows by reversing the steps above.
\end{proof}
\begin{remark}\label{rem3.1}
It must be noted that if the condition $y^{\alpha}b\cdot b^{\alpha}=(yz)^{\alpha}\cdot z$ in Theorem (\ref{thm3.3}) had been relaxed the middle generalized Bol identity would not have been realized. %The situation is similar in \cite{PS1} if the condition $yb\cdot b=(yz)\cdot z$ is relaxed in Proposition 3. 
In the absence of this condition, generalized Moufang loops with universal $\alpha$-elasticity property provide us with example of universally $\alpha$-elastic (flexible) law that is not middle generalized Bol. Thus, we have been able to, partly, answer the question (Is there a finite, universally flexible loop that is not middle Bol?) posed by Michael Kinyon at LOOPS '03, Prague 2003, in the generalized Bol context. Consequently, we suggest that the search for a finite, universally flexible loop that is not middle Bol should be intensified within loops with appropriate universal law of elasticity. If $\alpha$ is an identity map in Theorem \ref{thm3.3}, then $y^{\alpha}b\cdot b^{\alpha}=(yz)^{\alpha}\cdot z$ becomes $b\cdot b=z\cdot z$ for all $b,z$ which is just exponent $2$.   
\end{remark}
\begin{defn}
A loop $Q(\cdot)$ is called middle generalized Bol loop if it satisfies the identity \begin{equation}\label{MIGBL}(x/y)(z^{\alpha}\backslash x^{\alpha})=x(z^{\alpha}y\backslash x^{\alpha}).\end{equation}
\end{defn}
If the homomorphic self map $\alpha$ is an identity map, the middle generalized Bol identity (\ref{MIGBL}) reduces to middle Bol identity. 
\begin{corollary}
A Moufang loop is of exponent $2$ if and only if it is a middle Bol loop.
\end{corollary}
\begin{proof}
	Let $Q$ be a Moufang loop of exponent $2$. From the Moufang identity, $xy\cdot zx=x(yz\cdot x)$, we obtain $xy^{-1}\cdot z^{-1}x=x((z^{-1}y^{-1})^{-1}x)$ which implies $(x/y)(z\backslash x)=x((z^{-1}y^{-1})\backslash x)$
	since $Q$ is of exponent $2$, $y^{-1}=y$ for all $y\in Q$ and thus, $(x/y)(z\backslash x)=x((zy)\backslash x)$
	which gives the middle Bol identity.\\
	Conversely, let $Q$ be a middle Bol loop. Then,
	$(x/y)(z\backslash x)=x((zy)\backslash x)\implies$
	$xy^{-1}\cdot z^{-1}x=x((zy)^{-1}\cdot x)\implies$
	$xy^{-1}\cdot z^{-1}x=x(y^{-1}z^{-1}\cdot x)$.
	The last equation generalizes Moufang identity and on setting $y^{-1}=y$ for all $y\in Q$ we have $$xy\cdot zx=x(yz\cdot x)$$
	and thus, Moufang loop of exponent $2$.
%	which gives the Moufang identity and the condition $y^{-1}=y$ makes it of exponent $2$. 
\end{proof}	
\begin{remark}
The above corollary shows that search for a finite universally flexible loop (in Kinyon's parlance) that is not middle Bol should be intensified within Moufang loops of exponent different from $2$.	
\end{remark}
\section{The $\alpha$-alternative laws}\label{SEC:4}
In the generalized Moufang identity, $(xy\cdot z)y^{\alpha}=x(y\cdot zy^{\alpha})$ and the half-Moufang identity,  $(y^{\alpha}z\cdot y)x=y^{\alpha}(z\cdot yx)$, if we let $z=1$, we have $(xy)\cdot y^{\alpha}=x\cdot (yy^{\alpha})$ and  $(y^{\alpha}y)\cdot x=y^{\alpha}\cdot (yx)$ which give the right $\alpha$-alternative and left $\alpha$-alternative law, respectively. These new identities are called $\alpha$-alternative laws of generalized Moufang loop.

\begin{thm}\label{thm3.4}
A loop with universal $\alpha$-elasticity satisfies the left $\alpha$-alternative law $y^{\alpha}\cdot yz=y^{\alpha}y\cdot z$
if and only if it has the right inverse property.
\end{thm}
\begin{proof}
Let $e$ denote the unit element of the loop $Q(\cdot)$ and replace $x$ by $y^{-\alpha}$ and $z$ by $e$ in (\ref{UNI1}), then $$(yy^{\alpha})b^{-1}=y(b\backslash\left[(b/y^{-\alpha})b^{-1}\right])$$
If $Q(\cdot)$ satisfies the left $\alpha$-alternative law $$(yy^{\alpha})b^{-1}=y\cdot y^{\alpha}b^{-1}$$
$$y^{\alpha}b^{-1}=(b\backslash\left[(b/y^{-\alpha})b^{-1}\right])$$
$$b/y^{-\alpha}=by^{\alpha}$$
$$by^{\alpha}\cdot y^{-\alpha}=b$$
for every $y,b\in Q$ and $\alpha:Q\longrightarrow Q$.\\
Conversely if $Q(\cdot)$ has the right inverse property then $$b/y^{-\alpha}=by^{\alpha}$$
Using the identity $(yy^{\alpha})b^{-1}=y(b\backslash\left[(b/y^{-\alpha})b^{-1}\right])$, we have $$(yy^{\alpha})b^{-1}=y\left[b\backslash(by^{\alpha}\cdot b^{-1})\right]=y\left[b\backslash(b\cdot y^{\alpha}b^{-1})\right]=y\cdot y^{\alpha}b^{-1}$$
thus, $Q(\cdot)$ satisfies the left $\alpha$-alternative law.
\end{proof}

\begin{thm}\label{thm3.5}
A loop with universal $\alpha$-elasticity satisfies the right $\alpha$-alternative law $(xy)y^{\alpha}=x(yy^{\alpha})$ if and only if it has the left inverse property.
\end{thm}
\begin{proof}
Replace $z$ by $bz$ and $x$ by $e$ in (\ref{UNI2}) to obtain
$$x^{-1}(yy^{\alpha})=(\left[x^{-1}(y^{-1}\backslash x)\right]/x)y^{\alpha}$$
But, $$x^{-1}(yy^{\alpha})=x^{-1}y\cdot y^{\alpha}$$
Hence, $$ x^{-1}y=\left[x^{-1}(y^{-1}\backslash x)\right]/x$$
This implies that $$y^{-1}\backslash x=yx$$ and $$x=y^{-1}\cdot yx.$$
Conversely, If $Q(\cdot)$ has the left inverse property then \begin{align*}
x^{-1}(yy^{\alpha})&=(\left[(x^{-1}\cdot yx)\right]/x)y^{\alpha}&\\
&=(\left[(x^{-1}y\cdot x)\right]/x)y^{\alpha}&\\
&=(x^{-1}y)y^{\alpha}
\end{align*}
i.e $Q(\cdot)$ satisfies the right $\alpha$-alternative property.
\end{proof}
\begin{corollary}
A generalized Moufang loop with $\alpha$-universal elasticity is an inverse property loop if and only if it satisfies the left and right $\alpha$-alternative property.
\end{corollary}
\begin{proof}
The proof of this corollary follows from Theorem \ref{thm3.4} and \ref{thm3.5}.
\end{proof}
\begin{thm}\label{thm3.6}
A generalized Moufang loop, $Q(\cdot)$, with universal $\alpha$-elasticity property has the left inverse property if and only if it has the right inverse property.
\end{thm}
\begin{proof}
Let $Q(\cdot)$ be a RIP loop with universal $\alpha$-elasticity. Replace $z$ by $bz$ and $x$ by $e$ in (\ref{UNI2}) to obtain
\begin{equation}\label{IP}
by\cdot (b\backslash (bz\cdot y^{\alpha}))=(by\cdot z)y^{\alpha}
\end{equation}
if $y=(by)^{-1}$
\begin{equation}\label{IP1}
b\left[(by)^{-1}\cdot y^{\alpha}\right]=b(by)^{-1}y^{\alpha}
\end{equation}
Now let $z$ be replaced by $zx$ and $b$ be replaced by $e$ in (\ref{UNI1}),we obtain
\begin{equation}\label{IP2}
y(z\cdot y^{\alpha}x)=\left[(y\cdot zx)/x\right]\cdot y^{\alpha}x
\end{equation} 
and suppose $z=(y^{\alpha}x)^{-1}$ in (\ref{IP2}), then,
$$\left[y/y^{\alpha}x\right]x=y((y^{\alpha}x)^{-1}x).$$
put $y=b$, and $x=y$
\begin{equation}
\left[b/b^{\alpha}y\right]y=b((b^{\alpha}y)^{-1}y)
\end{equation}
$$b/b^{\alpha}y=b(b^{\alpha}y)^{-1}$$
$$b=b(b^{\alpha}y)^{-1}\cdot (b^{\alpha}y)$$
Hence, $$b=bx^{-1}\cdot x.$$
For every $y,b\in Q$ i.e $Q(\cdot)$ is a LIP-loop.\\
Conversely, Let $Q(\cdot)$ be a LIP-loop with universal $\alpha$-elasticity. Substitute $z=(y^{\alpha}x)^{-1}$ in identity (\ref{IP2}).
$$y=\left[(y\cdot (y^{\alpha}x)^{-1}x)/x\right]y^{\alpha}x$$
\begin{equation}\label{IP3}
\left[y(y^{\alpha}x)^{-1}\right]x=y(y^{\alpha}x)^{-1}x
\end{equation}
Put $z=(by)^{-1}$ in (\ref{IP}),
$$by\cdot (b\backslash (b(by)^{-1}\cdot y^{\alpha}))=y^{\alpha}$$
$$b\left[(by)^{-1}\cdot y^{\alpha}\right]=b(by)^{-1}y^{\alpha}$$
$$b\left[by\backslash y^{\alpha}\right]=b(by)^{-1}y^{\alpha}$$
Using the last identity and (\ref{IP3}), we get
$$b(by\backslash y^{\alpha})=b((by)^{-1}y^{\alpha})$$ Or 
$$b(b^{-1}y^{\alpha})=y^{\alpha}$$ for all $b,y\in Q$ and $\alpha:Q\longrightarrow Q$. i.e $Q$ is a RIP-loop,
\end{proof}
\begin{corollary}\label{cor1}
Let $Q(\cdot)$ be a generalized Moufang loop with universal $\alpha$-elasticity, then the following are equivalent
\begin{enumerate}
	\item $Q(\cdot)$ has right $\alpha$-alternative property
	\item $Q(\cdot)$ has left inverse property
	\item $Q(\cdot)$ has right inverse property
	\item $Q(\cdot)$ has left $\alpha$-alternative property
\end{enumerate}
\end{corollary}
\begin{proof}
The proof of this corollary follows from Theorem \ref{thm3.4}, Theorem \ref{thm3.5} and Theorem \ref{thm3.6}.
\end{proof}	
\begin{remark}
	What is customary in the theory of loops is that the right and middle nucleus coincide in right inverse property loops. Since generalized Moufang loops are right inverse property loops, the right and middle nucleus coincide. Despite the fact that generalized Moufang loops with universal $\alpha$-elasticity are right inverse property loops as proven in Theorem \ref{thm3.4}, the right and middle nucleus do not coincide. It must be noted that from Theorem \ref{thm3.2} the left and middle nucleus coincide, a clear deviation from what is obtainable in the theory of loops.
\end{remark}

\begin{remark}
One can infer from Theorem (\ref{thm3.6}) and Corollary (\ref{cor1}) that generalized Moufang loop with universal $\alpha$-elasticity property is an IP loop. It is known that all nuclei of IP loop coincide \cite{P}, such fact cannot be said of generalized Moufang loop with universal $\alpha$-elasticity property due to the statement of Theorem \ref{thm3.2}.
\end{remark}
\section{Associativity in the loops with universal $\alpha$-elasticity property}\label{SEC:5}
In what follows, we use the left and right $\alpha$-alternative laws of generalized Moufang loops with universal $\alpha$-elasticity property to establish that commutative IP-generalized Moufang loop with universal $\alpha$-elasticity property is an abelian group. The result is included in the following theorem.
\begin{thm}\label{CIPGBL}
Commutative IP-generalized Moufang loop with universal $\alpha$-elasticity property and $(y^{\alpha})^{2}\in N_{\rho}$ for all $y\in Q$ is associative.
\end{thm}
Before we prove Theorem \ref{CIPGBL}, we state the following lemma which is instrumental in the proof of the theorem.
\begin{lem}
Let $(Q,\cdot)$ be an IP-generalized Moufang loop with $\alpha$-universal elasticity property, then identities (\ref{UNI1}) and (\ref{UNI2}) are equivalent to 
\begin{equation}\label{EQUI}
(yx\cdot bz)x\cdot by^{\alpha}=yx\cdot b(zx\cdot by^{\alpha}).
\end{equation}
\end{lem}
\begin{proof}
Recall that in an IP-loop $x\backslash y=x^{-1}y$ and $x/y=xy^{-1}$. Using this, and replacing $y$ by $yx^{-1}$, $z$ by $b^{-1}z$, $x^{-1}$ by $x$, $b^{-1}$ by $b$ and $y^{\alpha}$ by $y^{\alpha}x^{-1}$ in (\ref{UNI1}) and on replacing $y$ by $b^{-1}y$, $z$ by $zx^{-1}$, $x^{-1}$ by $x$, $b^{-1}$ by $b$ and $y^{\alpha}$ by $b^{-1}y^{\alpha}$ in (\ref{UNI2}), we obtain
$$(yx\cdot bz)x\cdot by^{\alpha}=yx\cdot b(zx\cdot by^{\alpha}).$$
\end{proof}
We are now poised to give the proof of Theorem (\ref{CIPGBL}).
\begin{proof}
Replace $x$ by $y^{\alpha}x$ and $b$ by $by^{\alpha}$ in (\ref{EQUI}) and using the fact that $(y^{\alpha})^{2}\in N_{\rho}$, we obtain $$\left[x(by^{\alpha}\cdot z)y^{\alpha}x\right]b=x\left[by^{\alpha}\cdot (z\cdot y^{\alpha}x)\right]b.$$ If $z=x$, the equation above becomes $$\left(by^{\alpha}\cdot y^{\alpha}x\right)b=x\left(by^{\alpha}\right)^{2}.$$
Using commutativity of $Q(\cdot)$ and replacing $y^{\alpha}$ by $b^{-1}y^{\alpha}$ in the equation above, we get 
$$y^{\alpha}\left(b^{-1}y^{\alpha}\cdot x\right)\cdot b=x\left(y^{\alpha}\right)^{2}.$$
\begin{align*}
y^{\alpha}\left(b^{-1}y^{\alpha}\cdot x\right)&=x\left(y^{\alpha}\right)^{2}\cdot b^{-1}\\&=\left(y^{\alpha}\right)^{2}x\cdot b^{-1}=\left(y^{\alpha}\right)^{2}\cdot xb^{-1}=y^{\alpha}\left(y^{\alpha}\cdot xb^{-1}\right)\\&
b^{-1}y^{\alpha}\cdot x=y^{\alpha}\cdot xb^{-1}\\&
y^{\alpha}b^{-1}\cdot x=y^{\alpha}\cdot b^{-1}x
\end{align*}
Thus, the loop $Q(\cdot)$ is an abelian group. 
\end{proof}
\begin{thm}
The identity (\ref{EQUI}) implies both inverse properties in the generalized Moufang loop with universal $\alpha$-elasticity property.
\end{thm}
\begin{proof}
Let $b=e$ and $z=x^{-1}$ in (\ref{EQUI}), then 
\begin{align*}
(yx\cdot x^{-1})x\cdot y^{\alpha}&=yx\cdot y^{\alpha}&\\
yx\cdot x^{-1}&=y
\end{align*}
for all $x,y\in Q$. Hence, $Q(\cdot)$ is a RIP loop.
Also, let $x=e$ and $z=b^{-1}$
\begin{align*}
y\cdot by^{\alpha}&=y\cdot b(b^{-1}\cdot by^{\alpha})&\\
y^{\alpha}&=b^{-1}\cdot by^{\alpha}
\end{align*}
for every $y,b\in Q$ and $\alpha:Q\longrightarrow Q$. Hence, $Q(\cdot)$ is a LIP loop.
\end{proof}
\begin{corollary}
The generalized Moufang loop which satisfies the identity (\ref{EQUI}) is a loop with universal $\alpha$-elasticity property.
\end{corollary}

\bibliographystyle{apalike}

\end{document}